\RequirePackage{silence}
\documentclass[12pt,a4paper]{article}
\usepackage{amsfonts}
\usepackage{amsthm}
\usepackage{amsmath}
\usepackage{amscd}
\usepackage[latin2]{inputenc}
\usepackage{t1enc}
\usepackage[mathscr]{eucal}
\usepackage{indentfirst}
\usepackage{graphicx}
\usepackage{graphics}
\usepackage{pict2e}
\usepackage{epic}
\usepackage{mathtools}
\usepackage{amssymb}
\usepackage{wrapfig}
\usepackage{thmtools}
\usepackage{thm-restate}
\usepackage[hidelinks]{hyperref}
\usepackage{cleveref}

\numberwithin{equation}{section}
\usepackage[margin= 2.2 cm, bottom=19.5mm,footskip=12mm,top=20mm]{geometry}

\usepackage{epstopdf}

\theoremstyle{plain}
\newtheorem{Th}{Theorem}[section]
\newtheorem{Lemma}[Th]{Lemma}
\newtheorem{Cor}[Th]{Corollary}
\newtheorem{Prop}[Th]{Proposition}
\newtheorem{Cl}{Claim}

 \theoremstyle{definition}
\newtheorem{Def}[Th]{Definition}

\newtheorem{Rem}[Th]{Remark}
\newtheorem{?}[Th]{Problem}

\newtheorem*{Def*}{Definition}
\newtheorem*{Rem*}{Remark}

\providecommand{\keywords}[1]{\textbf{\textit{Keywords: }} #1}

\usepackage{tikz}
\usepackage{mathdots}
\usepackage{xcolor}
\usepackage{diagbox}
\usepackage{colortbl}
\usepackage[absolute,overlay]{textpos}
\usetikzlibrary{calc}
\usetikzlibrary{decorations.pathreplacing}
\usetikzlibrary{positioning,patterns}
\usetikzlibrary{arrows,shapes,positioning}
\usetikzlibrary{decorations.markings}

\tikzstyle{edge}=[very thick]

\title{The size-Ramsey number of short subdivisions}
\author{%
  Nemanja Dragani\'c \thanks{Department of Mathematics, ETH, 8092 Z\"urich, Switzerland. Email: \href{mailto:nemanja.draganic@math.ethz.ch} {\nolinkurl{nemanja.draganic@math.ethz.ch}}.}
\and Michael Krivelevich \thanks{School of Mathematical Sciences, Sackler Faculty of Exact Sciences, Tel Aviv University, Tel Aviv 6997801, Israel. Email: \href{mailto:krivelev@tauex.tau.ac.il}{\nolinkurl{krivelev@tauex.tau.ac.il}}. Supported in part by USA-Israel BSF grant 2018267, and by ISF grant 1261/17.}
\and 
Rajko Nenadov \thanks{Department of Mathematics, ETH, 8092 Z\"urich, Switzerland. Email: \href{mailto:rajko.nenadov@math.ethz.ch} {\nolinkurl{rajko.nenadov@math.ethz.ch}}.}
}

\begin{document}
\maketitle

\begin{abstract}
    The $r$-size-Ramsey number $\hat{R}_r(H)$ of a graph $H$ is the smallest number of edges a graph $G$ can have such that for every edge-coloring of $G$ with $r$ colors there exists a monochromatic copy of $H$ in $G$. 
    For a graph $H$, we denote by $H^q$ the graph obtained from $H$ by subdividing its edges with $q{-}1$ vertices each. In a recent paper of Kohayakawa, Retter and R{\"o}dl, it is shown that for all constant integers $q,r\geq 2$ and every graph $H$ on $n$ vertices and of bounded maximum degree, the $r$-size-Ramsey number of $H^q$ is at most $(\log n)^{20(q-1)}n^{1+1/q}$, for $n$ large enough.
    We improve upon this result using a significantly shorter argument by showing that  $\hat{R}_r(H^q)\leq O(n^{1+1/q})$ for any such graph $H$. 
\end{abstract}

\section{Introduction}
Given a graph $H$ and an integer $r\geq 2$, how few vertices can a graph $G$ have, so that however we color its edges with $r$ colors, we can find a monochromatic copy of $H$ in $G$? The answer to this question (which is indeed a finite number), is denoted by ${R}_r(H)$ and is called the $r$-Ramsey number of $H$. It is named after Frank P. Ramsey who was the first to study this question in his seminal paper \cite{ramsey2009problem}.
Following the given definition, we say that a graph $G$ is $r$-Ramsey for a graph $H$ and write $G\xrightarrow{}(H)_r$, if for any $r$-edge-coloring of $G$ there is a monochromatic copy of $H$ in $G$. Therefore we can write
\[R_r(H)=\min\left\{|V(G)|:G\xrightarrow{}(H)_r\right\}.\]
Note that if the minimum on the RHS is attained for a graph $G$, then it is also attained by a complete graph with $|V(G)|$ vertices. Can we find a graph which can have more vertices, but has fewer edges and is still $r$-Ramsey for $H$? How many edges suffice to construct a graph which is $r$-Ramsey for $H$? The answer to the latter question is called the $r$-\emph{size-Ramsey number} and is denoted by $\hat{R}_r(H)$:
\[\hat{R}_r(H)=\min\left\{|E(G)|:G\xrightarrow{}(H)_r\right\}.\]
This notion was introduced in 1978 by Erd\H{o}s, Faudree, Rousseau and Schelp \cite{erdHos1978size} and since then it has been the subject of extensive research.

Similar notions for measuring minimality of the host graph which is $r$-Ramsey for $H$ have also been studied. Some of them are Folkman numbers, chromatic-Ramsey numbers, degree-Ramsey numbers and so on. We refer the reader to a survey by Conlon, Fox and Sudakov \cite{conlon2015recent} for a recent thorough treatment of the topic. In this paper we will be concerned with $r$-size-Ramsey numbers.

\par  Beck \cite{beck1983size} showed that paths have linear 2-size-Ramsey number, or more precisely, he showed that for any sufficiently large $n$ we have $\hat{R}_2(P_n)\leq 900n$. His arguments can be easily extended to show that, more generally, it holds that $\hat{R}_r(P_n)=O_r(n)$ for any fixed $r\geq 2$. He also asked if $\hat{R}_2(H)$ grows linearly (in the number of vertices) for graphs with bounded maximum degree. This was proven to be true (for any constant number of colors) for trees by Friedman and Pippenger \cite{friedman1987expanding} and for cycles by Haxell, Kohayakawa and {\L}uczak \cite{haxell1995induced}.
In general, Beck's conjecture is not true, as R{\"o}dl and Szemer{\'e}di \cite{rodl2000size} showed that there exists a constant $c > 0$ such that for every sufficiently large $n$ there exists a graph $H$ with $n$ vertices and maximum degree $3$ for which $\hat{R}_2(H) \geq  n \log^ c n$. 

When it comes to more precise bounds, depending on the number of colors, Dudek and Pralat \cite{dudek2017some} showed that $\hat{R}_r(P_n)\geq \Omega(r^2)n$, and this is almost optimal as Krivelevich \cite{krivelevich2019long} showed that $\hat{R}_r(P_n)\leq O(r^2\log r)n$. In subsequent papers \cite{bal2019new,dudek2018note} their bounds are improved by constant factors. Moreover, when $r=2$, the bounds were gradually improved in a series of papers, see \cite{bal2019new, beck1983size, bollobas1986extremal, dudek2017some} for lower bounds, and \cite{beck1983size,bollobas2001random,dudek2015alternative,dudek2017some, letzter2016path} for upper bounds. The current best bounds  are $(3.75-o(1))n\leq \hat{R}_2(P_n) \leq 74n$. For results on size-Ramsey numbers of powers of paths, see \cite{clemens2019size,han2018multicolour}, and for graphs with bounded treewidth, see the recent paper \cite{kamcev2019}. For results concerning general trees, see the paper by Dellamonica \cite{dellamonica2012size}.

 The mentioned result for cycles by Haxell, Kohayakawa and {\L}uczak uses the regularity lemma and makes no attempt to optimize the constant. Their result was improved by Javadi, Khoeini, Omidi and Pokrovskiy \cite{javadi2019size} by using different arguments. In particular, they prove that $\hat{R}_2(C_n)\leq 843\times 10^6n$, for $n$ large enough. 

\subsection{Size-Ramsey numbers of subdivisions of graphs}

At the moment, a satisfactory result about size-Ramsey numbers of general bounded degree graphs seems out of reach. A natural step in this direction is understanding subdivisions of bounded degree graphs. For a graph $H$, let $H^q$ be the graph obtained from $H$ by subdividing each of its edges with $q-1$ vertices (instead of each edge, we have a path of length $q$). Pak \cite{pak2002mixing} conjectured that for a given graph $H$ with $n$ vertices and constant maximum degree, and for $q=\Omega(\log n)$, the $2$-size-Ramsey number of $H^q$ is linear in $|V(H^q)|$. By using random walks on expanders, he was able to show this up to a polylogarithmic factor. In the special case where $H$ is fixed and $q=q(n)$ grows with $n$, the conjecture was resolved by Donadelli, Haxell and Kohayakawa \cite{donadelli2005note}.

On the other hand, complementing the work of Pak, size-Ramsey numbers of short subdivisions of bounded degree graphs were studied in a recent paper by Kohayakawa, Retter and R\"odl \cite{kohayakawa2019size}. They showed that if $q$ and $r$ are constant, then $\hat{R}_r(H^q)\leq (\log n)^{20(q-1)}n^{1+1/q}$ for large enough $n$. In fact, they prove a universality result, which is usually much harder than finding one fixed monochromatic copy in the graph we color. To state it, they use the following definition:

\begin{Def*}(Universal size-Ramsey number).
For positive integers $D, q,r,n$, let the \emph{universal size-Ramsey number} $USR(D,q,r,n)$ be the smallest number of edges a graph $G$ can have such that 
\[G\xrightarrow{}(H^q)_r\qquad \text{for all graphs } H \text{ on } n \text{ vertices and } \Delta(H)\leq D.\]
\end{Def*}

Kohayakawa, Retter and R\"odl show that $USR(D,q,r,n)\leq (\log n)^{20(q-1)}n^{1+1/q}$, for large enough $n$. They also obtain a lower bound:
    \[n^{1+1/q-2/(Dq)+o(1)}\leq USR(D,q,1,n)\leq USR(D,q,r,n).\]
which shows that their result is almost tight. They conjecture that the power of the logarithm in the upper bound can be substantially reduced, so that it does not depend on $q$, or even that it can be entirely removed.
  
\par In this paper we remove the polylogarithmic factor from their upper bound on the universal size-Ramsey number, thus also improving their bound on the size-Ramsey number of short subdivisions. We do so by using a significantly shorter argument. Our claim follows directly from the next theorem which we prove in the third section, after giving some preliminary results in Section 2. For proving upper bounds on size-Ramsey numbers very often random graphs play an important role; we show that random graphs with appropriate parameters are $r$-Ramsey for $H^q$.  
\begin{restatable}[]{Th}{main}\label{main}
Let $D,q,r\geq 2$ be positive integers. There exist positive constants $c,\mu_0$ such that whp for $p=cn^{-1+1/q}$ the random graph $G\sim G(n,p)$ is $r$-Ramsey for $H^q$, for every graph $H$ with $\mu_0n$ vertices and maximum degree $D$.
\end{restatable}
\begin{Rem*}
A simple first moment argument shows that this result is almost optimal in terms of $p$, as for $p\ll n^{-1+1/q-2/(Dq)}$ and for a $D$-regular graph $H$ on $\mu_0 n$ vertices, $G\sim G(n,p)$ does not contain $H^q$ whp. Indeed, the expected number of labeled copies of such a graph $H^q$ is less than $n^{n_0}p^{e_0}$, where $n_0=\frac{\mu_0nD}{2}(q-1)+\mu_0 n$ and $e_0=\frac{\mu_0 nD}{2}q$ are the number of vertices and edges of $H^q$ (respectively). When $p$ is small as mentioned, the expectation tends to zero, hence by Markov's inequality there is no copy of $H^q$ in $G(n,p)$ whp. In other words, there is a small gap in terms of $p$ between containing a single fixed copy of a graph from our class, and being $r$-Ramsey for all graphs in the class.
\end{Rem*}

Since $G\sim G(n,p)$ has at most $2\binom{n}{2}p$ edges whp, for $p$ as specified in \Cref{main}, we obtain the following result:

\begin{Cor} 
For all $D,q,r,n$ and every graph $H$ on $n$ vertices and with $\Delta(H)<D$, it holds: 
\[\hat{R}_r(H^q)\leq USR(D,q,r,n)\leq cn^{1+1/q}\]
for a constant $c=c(D,q,r)$.
\end{Cor}
\textbf{Notation.}  Let $G=(V,E)$ be a graph and $V_1, V_2\subseteq V$ and $E_1\subset E$. We denote by $G[V_1]$ the subgraph of $G$ induced by $V_1$, by $G[E_1]$ the graph $(V,E_1)$, and by $N_{V_2}(V_1)$ the set of vertices in $V_2$ which are adjacent to at least one vertex in $V_1$. We also use the notation $N_G(V_1):=N_V(V_1)$ (we omit $G$ in the subscript when it is unambiguous) and $N_{V_2}(v):=N_{V_2}(\{v\})$ for $v\in V$. We denote by $G(n,p)$ the binomial random graph, i.e.\ the probability space of graphs on n vertices where each pair of vertices forms an edge independently with probability $p$.
We say that $G(n,p)$ satisfies a property with high probability (whp) if a graph sampled from $G(n,p)$ satisfies this property with probability tending to $1$ as $n$ tends to infinity. We use standard Landau notations $O(\cdot),o(\cdot), \omega(\cdot), \Omega(\cdot)$. With $\log n$ we denote the natural logarithm of $n$. We omit floors and ceils whenever they are not essential. For integers $m,n$ we write $n\gg m$ when we want to say that $n$ is large enough in comparison with $m$, but the exact dependency is not essential.

\section{Preliminaries}

Our proof will depend on finding many random-like bipartite graphs in a monochromatic subgraph of the graph we color. We describe what it means to be random-like in the following two lemmas, through the notion of regularity.
\begin{Def*}
    Given a graph $G$ and disjoint subsets $U, W \subset V (G)$, we
say that the pair $(U, W)$ is $(G, \varepsilon, p)$-\emph{regular} for some $\varepsilon,p \in (0, 1)$ if 
$$ 
    \Big| \frac{e_G(U,W)}{|U||W|}-\frac{e_G(U',W')}{|U'||W'|}\Big|<\varepsilon p
$$
for every  $U'\subset U$ of size $|U'| \geq \varepsilon|U|$, and $W'\subset W$ of size $|W'| \geq \varepsilon|W|$. If $G=(U\cup W,E)$ is bipartite and $(U, W)$ is $(G, \varepsilon, p)$-regular, then we say that $G$ is an $(\varepsilon,p)$-regular pair.
\end{Def*}

\begin{Def*}
    Given a bipartite graph $(U\cup W,E)$ we say that $U$ and $W$ form an $(\varepsilon,p)$-\emph{lower-regular} pair if 
    \[
         \frac{e(U',W')}{|U'||W'|}\geq (1-\varepsilon)p
    \]
    for every  $U'\subset U$ of size $|U'| \geq \varepsilon|U|$ and every $W'\subset W$ of size $|W'| \geq \varepsilon|W|$.
\end{Def*}
\noindent
Note that if $(U,W)$ is $(G,\varepsilon,p)$-regular and $e(U,W)=|U||W|p$ then $U$ and $W$ also form an $(\varepsilon,p)$-lower-regular pair in $G$.

Now we state a standard result which tells us that whp for all colorings of a random graph there exists a nicely structured monochromatic subgraph.
\begin{Prop}\label{Graph}
    Let $h>0$. For any $ \varepsilon>0$ and integers $r, K \ge 1$, there exist $0<\mu=\mu(r,K,\varepsilon), c'=c'(r,K)<1/2$ such that if $p \geq n^{-1+h}$ then $G = G(n,p)$ whp has the following property.
    
    Let $E(G) = E_1 \cup \ldots \cup E_r$ be an $r$-edge-coloring of $G$. Then, for some $i \in \{1, \ldots, r\}$, there exists a subgraph $G'\subseteq G[E_i]$ and disjoint subsets $V_1, \ldots, V_K \subseteq V(G')$ such that the following holds:
    \begin{itemize}
        \item $|V_i| = \mu n$ for each $i \in [K]$,
        \item $(V_i, V_j)$ forms an $(G', \varepsilon, p')$-regular pair for every distinct $i, j \in [K]$, where $p'= c' p$,
        \item $e_{G'}(V_i, V_j) = (\mu n)^2 p'$ for every $i,j\in[k]$, and
        \item $|N(v) \cap V_i| \le \tfrac{5}{4} \mu n p'$ for every $v \in \bigcup_{j \in [K]} V_j$ and $i \in [K]$.
    \end{itemize}

 \end{Prop}
The proof of the above proposition relies on a standard argument invoking a sparse version of Szemeredi's Regularity lemma, finding the wanted monochromatic configuration by applying Turan's theorem and Ramsey's theorem, and then "cleaning up" the graph to satisfy all of the given conditions. See, for example, \cite{haxell1995induced}.

 The following definition and the lemma which builds on it can be found in \cite{gerke2007small}.
 
 \begin{Def}\label{def}
    Let $P_\ell(n, m, \varepsilon)$ be the set of all graphs consisting of  pairwise disjoint sets of vertices  $V_1,...,V_{\ell}$ of size $n$ such that for $i \in[\ell]$, the sets $V_i, V_{i+1}$ form an $(\varepsilon, m/n^2)$-lower-regular graph with $m$ edges\footnote{With the convention that $V_{\ell+1}=V_1$.}. We call elements in $P_\ell(n, m, \varepsilon)$ \emph{chains}. Let $0<\nu<1$. We say that a set $Q \subset V_1$ is $(1 - \nu)$-\emph{spanning} in the chain if all but at most $\nu n$ vertices in $V_1$ can be reached by paths of length $\ell$ which start in $Q$ and then go through $V_2, V_3,\ldots,V_\ell$ in this order. Another way to write this is:
    $$|N_{V_1}(N_{V_\ell}(N_{V_{\ell-1}}(...N_{V_2}(Q)))|>(1 - \nu)n.$$
    We call a chain in $P_\ell(n, m, \varepsilon) $ \emph{expanding} with respect to $\delta, \gamma,\nu,C$ if it contains a set $X \subset V_1$ of size at most $\delta n$ so that for all $t \geq Cn^{\ell+1}/m^{\ell}$ at most $\gamma^t\binom{n}{t}$ sets of size $t$ in $V_1 - X$ are not $(1 -\nu)$-spanning.
\end{Def}

\begin{Rem}\label{singletons}
    If a chain in $P_\ell(n, m, \varepsilon) $ is expanding with respect to $\delta, \gamma,\nu,C$ and it holds that $Cn^{\ell+1}/m^{\ell}<1$, then at most $(\delta+\gamma)n$ singletons in $V_1$ are not $(1-\nu)$-spanning.
\end{Rem}

\begin{Lemma}\label{Angelika}\emph{(Lemma 5.9 in \cite{gerke2007small})}\,
Let  $\ell> 2$ be an integer, and let $0 < \beta,\delta,\gamma,\nu < 1/3$. Then there exist an $\varepsilon_1 =
\varepsilon_1 (\ell,\beta,\delta,\gamma,\nu) > 0$ and a constant $C = C(\ell, \nu)$ such that for all $0 < \varepsilon\leq\varepsilon_1$, the number of
graphs in $P_\ell(n, m, \varepsilon)$ that are expanding with respect to $\delta,\gamma,\nu$ and $C$ is at least $$ (1-\beta ^m)\binom{n^2}{m}^{\ell} $$
for all $m\geq 8n\log n$. 
\end{Lemma}

A simple consequence of the previous result is that typically all chains in a random graph with appropriate parameters are expanding. 

\begin{Cor}\label{angelika2}
    Let  $\ell \geq 2$ be an integer, and let $c'>0$ and $0 <\delta,\gamma,\nu,\mu< 1/3$. Then there exist an $\varepsilon_0 =\varepsilon_0 (\ell,\delta,\gamma,\nu,c') > 0$ and a constant $C = C(\ell, \nu)$ such that for any $0 < \varepsilon<\varepsilon_0$ the following holds whp. Any chain in $ P_\ell(\mu n, m, \varepsilon)$ which is a subgraph of $G(n,p)$ for $p=\omega(\log n/n)$ and $m\ge (\mu n)^2c'p$ is  expanding with respect to $\delta,\gamma,\nu,C$.

\end{Cor}

\begin{proof}
Let $\varepsilon_0=\varepsilon_1(\ell,\beta,\delta,\gamma,\nu)$ and $C=C(\ell,\nu)$ be given by Lemma \ref{Angelika} for the given parameters with the same name and $\beta=\beta(\ell,c')$ chosen small enough - we will see later how small.

 
 For a fixed $m$ larger than $m_0:=(\mu n)^2c'p$, we now bound the probability that a fixed $\ell$-tuple $(V_1,V_2,..., V_{\ell})$ of sets of size $\mu n$ in $G(n,p)$ induces a graph containing a chain from $P_{\ell}(\mu n,m, \varepsilon)$ which is not $(1-\nu)$-expanding. For this we use a union bound over all such non-expanding chains. An upper bound for the number of non-expanding chains is given by Lemma \ref{Angelika}, so we have that the probability in question is bounded by 
 \begin{equation*}
    \beta ^m\binom{(\mu n)^2}{m}^{\ell} p^{m\ell}<\beta ^m\Big(\frac{e(\mu n)^2}{m}\Big)^{m\ell}p^{m\ell}\leq 
    \beta^m\Big(\frac{e (\mu n)^2}{(\mu n)^2c'p}\Big)^{m\ell}p^{m\ell}\leq \Big(\frac{e}{c'}\Big)^{m\ell}\beta^m.
  \end{equation*}
 Now we easily bound the probability that there is a non-expanding chain for some $m\geq m_0$ and induced by any $\ell$-tuple $\zeta=(V_1,V_2,...,V_{\ell})$:
 \[
    \sum_m\sum_\zeta \Big(\frac{e}{c'}\Big)^{m\ell}\beta^m\leq
    \sum_m \ell^n\Big(\left(\frac{e}{c'}\right)^\ell\beta\Big)^{m}\leq
    n^2\ell^n\Big(\left(\frac{e}{c'}\right)^\ell\beta\Big)^{m_0}
 \]
 which tends to $0$ when $\beta$ is chosen small enough, as $m_0$ is super-linear in $n$, while $e,c'$ and $\ell$ are constants.
\end{proof}

The following lemma follows easily from Chernoff's inequality.
\begin{Lemma}\label{Chernoff}
    Let $0<\alpha<1$ and $G\sim G(n,p)$ with $p=cn^{-1+1/q}$ for  constants $c,q\geq 1$. Then whp for every set $S$ such that $|S|\leq \alpha n$, there are at most $O(n^{1-1/q})$ vertices in $V(G)$ with more than $4\alpha np$ neighbors in $S$.
\end{Lemma}

\begin{proof}
Fix $S,R\subseteq V(G)$ of sizes $|S|=\alpha n$ and $|R|=Cn^{1-1/q}$, for a constant $C>2/\alpha$. The expected number of edges with one vertex in $R$ and the other in $S$ is at least $|R||S|p/2$ and at most $|R||S|p$. If $R$ is such that all of its vertices have at least $4\alpha np$ neighbors in $S$ then there are at least $\frac{4\alpha np|R|}{2}=2|R||S|p$ edges with one vertex in $R$ and one in $S$, where we divide by 2 because of possible double counting when a part of $R$ is in $S$. Therefore, the probability that a fixed set $R$ is as described, i.e.\ that $e(R,S)$ is at least two times larger than its expectation, is by Chernoff's inequality at most $e^{-|S||R|p/6}=e^{-\alpha cCn}\leq e^{-2n}$. Summing over all sets $S$ and $R$ we get that the probability of an unwanted event is bounded by $2^n2^ne^{-2n}$, which tends to $0$, so we are done.
\end{proof} 


\section{The Proof}
Before we prove our main result, which is Theorem \ref{main}, we will give some useful definitions and lemmas which should make the proof of the main result quite straightforward. The ultimate goal is to prove that $G(n,p)$ is whp $r$-Ramsey for $H^q$, for suitably chosen parameters.
\begin{Def}
Given a graph $G$ and disjoint subsets $A, B \subset V(G)$, we say that a vertex $v\in A$
is $(A, B, q, \nu)$-\emph{expanding}, for some $q\in \mathbb{N}$ and $\nu> 0$, if for at least $(1-\nu)|A|$ vertices $w\in A$ there
exists a $v-w$ path in $G$ of length $q$ and with all $q-1$ internal vertices being in B.
\end{Def}

\begin{Def}\label{goodness}
    Let $G$ be a graph and $A\cup B$ be a partition of its vertex set. We say that the pair $(A,B)$ is $(q,\nu, n, c)$-\emph{good} in $G$ if:
\begin{itemize}
    \item[(1)] $|A|=n$.
    \item[(2)]  All vertices in $A$ are $(A, B, q, \nu)$-expanding in $G$. 
    \item[(3)] $\Delta(G)<2np$ for $p=cn^{-1+1/q}$.
    \item[(4)] For every $B'\subset A\cup B$ of size $|B'|<\frac{n}{c^{q+2}}$ there are at most $\frac{n^{1-1/(2q)}}{q}$ vertices $v\in A\cup B$ with 
    $$
    |N_{B'}(v)|\geq \frac{2np}{c^{q+3/2}}.
    $$ 
\end{itemize}
\end{Def}
\begin{Def}
    Let $G$ be a graph and let $v\in V(G)$. We define the $k^{\text{th}}$ \emph{neighborhood} of $v$ in $G$ as $$ N^k_G(v)=\underbrace{N_G(N_G(\ldots (N_G}_{k \text{ times}}(v))\ldots)),$$
    i.e.\ the set of vertices reachable from $v$ by walks of length exactly $k$. We omit $G$ in the subscript whenever it is unambiguous.
\end{Def}

The next lemma tells us that in a good pair $(A,B)$ there are not many vertices in $A$ for which there exists a $k\leq q-1$ such that their $k^{\text{th}}$ neighborhood has large intersection with a fixed small set in $B$.
    
    \begin{Lemma}\label{B'}
     Let $(A,B)$ be a $(q,\nu, n, c)$-good pair in $G$ where $n\gg c\gg q$.
     For every $B'\subset B$ of size $|B'|<\frac{n}{c^{q+2}}$ there are at most $n^{1-1/(2q)}$ bad vertices $v\in A$, i.e.\ vertices with the following property:
     $$ |N_G^k(v)\cap B'|\geq \frac{(np)^k}{c^{q+1}} $$
     for some $1\leq k\leq q-1$.
    \end{Lemma}
\begin{proof}
Let $B'\subset B$ be of size $|B'|<\frac{n}{c^{q+2}}$. Define $X_1$ to be the set of all vertices in $A\cup B$ which have more than $\frac{2np}{c^{q+3/2}}$ neighbors in $B'$. Now define $X_{k+1}$ to be the set of vertices in $A\cup B$ with more than $\frac{2np}{c^{q+3/2}}$ neighbors in $X_k$, for all $k\in [q-2]$. 

\par It will be enough to prove that each set $X_k$ is small and that all vertices in $A \setminus X_k$ have less than $\frac{(np)^k}{c^{q+1}}$ vertices from their $k^{\text{th}}$ neighborhood in $B'$, so the set of bad vertices will be a subset of $X_1\cup\ldots\cup X_{q-1}$.  Fix $k\in [q-1]$. 
 \begin{Cl}\label{claim1}
 $|X_k|<\frac{n^{1-1/(2q)}}{q}$.
 \end{Cl}
 Note that condition $(4)$ of Definition \ref{goodness} implies that $|X_1|<\frac{n^{1-1/(2q)}}{q}$. Similarly, using induction, $|X_i|<\frac{n^{1-1/(2q)}}{q}$ for all $i\in[k]$.
 
 \begin{Cl}\label{claim2}
    For all vertices $v\in A\setminus X_k$ it holds that $|N^k(v)\cap B'|<\frac{(np)^k}{c^{q+1}}$.
 \end{Cl}
 Let $v\in A\setminus X_k$. Since $v\not\in X_k$, it has at most $\frac{2np}{c^{q+3/2}}$ neighbors in $X_{k-1}$. Therefore, by using paths which have their second vertex in $X_{k-1}$, $v$ can reach at most $\frac{2np}{c^{q+3/2}} \cdot (2np)^{k-1}$ vertices in $B'$ in exactly $k$ steps, as by definition of a good pair $\Delta(G)<2np$. All other neighbors of $v$ are not in $X_{k-1}$, so this means that all of them have at most $\frac{2np}{c^{q+3/2}}$ neighbors in $X_{k-2}$. Therefore, there are at most 
 $$
    2np \cdot \frac{2np}{c^{q+3/2}} \cdot (2np)^{k-2}
 $$
 other vertices passing through $X_{k-2}$ in the third step and then finishing in $B'$ after $k$ steps. We continue in this fashion and we get the following upper bound on the number of vertices in $B'$ in the $k^{\text{th}}$ neighborhood of $v$:
 \[
    \sum_{i=1}^{k} (2np)^{i-1}\cdot \frac{2np}{c^{q+3/2}} \cdot (2np)^{k-i}=\sum_{i=1}^{k}\frac{(2np)^{k}}{c^{q+3/2}}<\frac{(np)^k}{c^{q+1}}
 \]
 for $c$ chosen large enough in the beginning.
\par To complete the proof notice that all vertices $v\in A\setminus(X_1 \cup \ldots \cup X_{q-1})$ have the property  $$ |N^k(v)\cap B'|< \frac{(np)^k}{c^{q+1}}$$ for all $k\in [q-1]$ due to \Cref{claim2}, and by \Cref{claim1} that $|X_1\cup\ldots\cup X_{q-1}|\leq n^{1-1/(2q)}$,
so we are done, as all the bad vertices live in the small set $X_1 \cup \ldots \cup X_{q-1}$.
\end{proof}

The following lemma shows that if a pair $(A,B)$ is good then after deleting a certain relatively small number of vertices in $B$, most of the vertices in $A$ are still expanding for suitably chosen parameters.     
\begin{Lemma}\label{robust}
For every integer $q$ and positive $\nu>0$ and for $n\gg c\gg q,1/\nu$, the following holds. Let
$G$ be a graph and $A, B \subset  V(G)$ be such that $(A,B)$ is $(q,\nu, n, c)$-good in $G$. Then all but at most $n^{1-1/(2q)}$ vertices in $A$ are $(A,B-B',q,2\nu)$-expanding, for every $B'\subset B$ of size $|B'|<\frac{n}{c^{q+2}}$.
\end{Lemma}
\begin{proof}
  Let $B'\subseteq B$ be of size $|B'|<\frac{n}{c^{q+2}}$. Let $X\subset A$ be the set of bad vertices described in Lemma \ref{B'} for the set $B'$, so that $|X|\leq n^{1-1/(2q)}$. We prove that all vertices in $A-X$ are $(A, B-B',q, 2\nu)$-expanding.
    \par Let $v\in A-X $. Thanks to condition $(3)$ from Definition \ref{goodness} we have that $\Delta(G)<2np$, and from Lemma \ref{B'} we know that there are at most $\frac{(np)^k}{c^{q+1}}$ vertices in $N_G^k(v)\cap B'$ for each $k\in[q-1] $. Using these two facts we obtain an upper bound for the number of vertices in $A$ reachable from $v$ by paths of length $q$ which contain vertices in $B'$:
    \begin{equation*}
        \sum_{k=1}^{q-1}\frac{(np)^k}{c^{q+1}}(2np)^{q-k}\leq (q-1)\frac{(2np)^q}{c^{q+1}}=\frac{2^q (q-1)}{c}n
    \end{equation*}
    where the $k^{\text{th}}$ term in the sum is an upper bound on the number of vertices in $A$ which can be reached from $v$ by paths whose $k^\text{th}$ vertex is in $B'$. If we choose $c$ so that $2^q(q-1)/c<\nu$, we are done as now at most $\nu|A|$ new vertices in $A$ start being unreachable by removing $B'$, i.e.\ $v$ is $(A,B-B',q,\nu+\nu)$-expanding.
\end{proof}

We finish our preparation for the proof of the main theorem by showing that an appropriate combination of good pairs contains the subdivisions of bounded degree graphs of linear size.
\begin{Lemma}\label{algorithm}
For every two integers $q,D$ and integers $n\gg c\gg q,D$ the following holds. Let $G$ be a graph and $A, B_1, . . . , B_D \subset  V(G)$ be disjoint subsets such that for each $i\in [D]$ we have that $(A,B_i)$ is 
$(q,1/4D, n, c)$-good in $G[A\cup B_i]$. Then $G$ contains a copy of the $q$-subdivision $H^q$ of any graph $H$ with at most $n/c^{q+4}$ vertices and $\Delta(H)\leq D$.
\end{Lemma}

\begin{proof}
Let $n\gg c$ be large enough integers which we get from Lemma \ref{robust} for 
$\nu:=1/4D.$ We will embed $H^q$ into $G$ by embedding the vertices of $H$ one by one into $A$, and connecting them to all previously embedded neighbors from $H$ by paths of length $q$ whose internal vertices go through different $B_i$. Lemma \ref{robust} will tell us that the small number of vertices in each $B_i$ which are used during this embedding process will not be able to prevent us from finding new paths for the new vertices that we embed. We construct an embedding $\varphi:H^q\hookrightarrow G$ as follows. For each $x\in V(H^q)$ we denote by $\Bar{x}$ the image of $x$, i.e.\ $\Bar{x}=\varphi(x)$.\\
\begin{itemize}
    \item [1)] Let $S:=\varnothing$ be the set of \emph{occupied} vertices in $A$ and $S_t:=\varnothing$ the set of \emph{occupied} vertices in $B_t$ for each $t\in [D]$. Set $Y:=\varnothing$ to be the set of vertices from $V(H)$ which are already embedded. 
    \item  Repeat Steps 2, 3 and 4 until $Y=V(H)$:\\
  
    \item [2)] Take any $v\in V(H)-Y$. If $v$ does not have any already embedded neighbors, then let $\Bar{v}$ be an arbitrary vertex in $A-S$. Set $Y:=Y\cup \{v\}$, $S:=S\cup \{\Bar{v}\}$. 
    
     \item [3)] Else (if $v$ has at least one already embedded neighbor), let $N_v$ be the set of those neighbors, i.e.\ vertices adjacent to $v$ in $H$ which are in $Y$. Assign a distinct integer $t(w)\leq D$ to each vertex $w\in N_v$. Find a vertex $a\in A-S$ such that
     for every $w\in N_v$ there is a path $P=P_{t(w)}$ with the following properties\footnote{Later we show that this is possible.}:
     
     \begin{itemize}
        \item $P$ is of length $q$; it starts at $\Bar{w}$ and ends in $a$, while all its internal vertices in $B_{t(w)}$.
         \item None of the vertices of $P$ is occupied, i.e.\ $S_{t(w)}\cap V(P)=\varnothing$.
     \end{itemize}
     Now set $$\Bar{v}:=a,\qquad Y:=Y\cup \{v\}, \qquad S:=S\cup \{a\}, \qquad  S_{t(w)}:=S_{t(w)}\cup V(P_{t(w)})-\{a,\Bar{w}\},$$ for each $w\in N_v$.
     \item [4)] If there exists a vertex $v\in Y$ such that $\Bar{v}$ is not $(A,B_i-S_i,q,2\nu)$-expanding for some $i\in[D]$, then do $Y:=Y-\{v\}$ for every such vertex, i.e.\ we discard it from the set of already embedded vertices. \textit{Notice that we do not delete} $\Bar{v}$ \textit{from} $S$, \textit{meaning that this vertex in $G$ is now marked as \emph{useless} and cannot be used for still non-embedded vertices. The same holds for all the vertices which were lying on paths starting at $\Bar{v}$, as they also stay \emph{occupied}}. 
\end{itemize}

In order to complete the proof we need to prove that it is always possible to carry out step $3$ of the algorithm, and that the algorithm terminates at some point.

For the former task, note that when we enter the repeat loop, the images of all vertices in $Y$ are $(A,B_i-S_i,q,2\nu)$-expanding for every $i\in[D]$, thanks to step $4$ of the algorithm. This means that at most $2\nu|A|=\frac{|A|}{2D}$ vertices in $A$ are not reachable by paths starting at $\Bar{w}$ with all internal vertices in $B_{t(w)}-S_{t(w)}$, for each already embedded neighbor $w\in N_v$. Furthermore, as we will prove in the next paragraph, we enter the repeat loop at most $|A|/c^{q+3}$ times, so the number of occupied vertices in $A$ is also at most $|A|/c^{q+3}$. Recalling that $|N_v|\leq D$, we conclude that there are at most $D\cdot\frac{|A|}{2D}+\frac{|A|}{c^{q+3}}<|A|$ occupied or non-reachable vertices, hence there is a non-occupied vertex $a\in A$ connected by paths (as described in step 3) going through $B_{t(w)}-S_{t(w)}$ to each $\Bar{w}$. \par 
To prove that the algorithm terminates, we will show that we are done after entering the repeat loop at most $n/c^{q+3}$ times. For the sake of contradiction, assume that after $n/c^{q+3}$ steps, there are still vertices in $V(H)$ which are not embedded, i.e.\ $|Y|<|V(H)|\leq n/c^{q+4}$. Since $S$ increases by $1$ each time we enter the repeat loop, the number of vertices which are discarded in step $4$ is at least $|S|-|Y|>n/c^{q+3}-n/c^{q+4}>Dn^{1-1/(2q)}$ for $n$ chosen large enough in the beginning. Note that vertices are discarded if they are not $(A,B_i-S_i,q,2\nu)$-expanding and furthermore, that they continue to be non-expanding as the size of $S_i$ is non-decreasing. But by Lemma \ref{robust}, the number of vertices in $A$ which are not $(A,B_i-S_i,q,2\nu)$-expanding is at most $n^{1-1/(2q)}$ for every $i$ because $|S_i|\leq q|S|<\frac{qn}{c^{q+3}}<\frac{n}{c^{q+2}}$, so there are at most $Dn^{1-1/(2q)}$ vertices in $A$ which could have been discarded, a contradiction.
\end{proof}
\subsection{Proof of \Cref{main}}
Now we put everything together to prove our main result. Thanks to \Cref{algorithm} it will be enough to show the existence of a structure described in the lemma, with the appropriate parameters. For the existence we will use \Cref{Graph} and \Cref{angelika2}.
\main*

\begin{proof}
In order to use \Cref{Graph}, we first fix the relevant parameters: $h=1/q$, $K:=(q-1)D+1$. Now we specify $\varepsilon$. Note that $c'=c'(r,K)$ in \Cref{Graph} is already determined, so we can choose $\varepsilon$ so that $0<\varepsilon<\varepsilon_0(q,\delta,\gamma,\nu,c')$ given by Corollary \ref{angelika2} where $\delta=\gamma=\nu=1/8D$. Let also $C=C(q,\nu)$ from \Cref{angelika2}. Finally, let $p=cn^{-1+1/q}$, where $c$ is a large enough constant.

 Take any $r$-coloring of $G$. We want to prove that there is a color class which contains a copy of $H^q$. Thanks to Proposition \ref{Graph} we whp get $K=(q-1)D+1$ sets 
 $$\{V_1\}\cup\Big\{V_k^t \mid k\in \{2,...,q\},\ t\in [D]\Big\}$$ 
 such that for each fixed $t\in [D]$ the sets  
 $\{V_1\}\cup\Big\{V_k^t \mid k\in \{2,...,q\}\Big\}$ form a chain in $P_q(\mu n, (\mu n)^2p', \varepsilon)$ of the same color for parameters as specified in the proposition. Notice that the proposition tells us that each pair of sets among these $K$ sets satisfies some properties, but we only use the edges of the pairs which make this star-like configuration, like in Figure 1. This gives us $D$ chains which all share the first set of vertices $V_1$. 
 

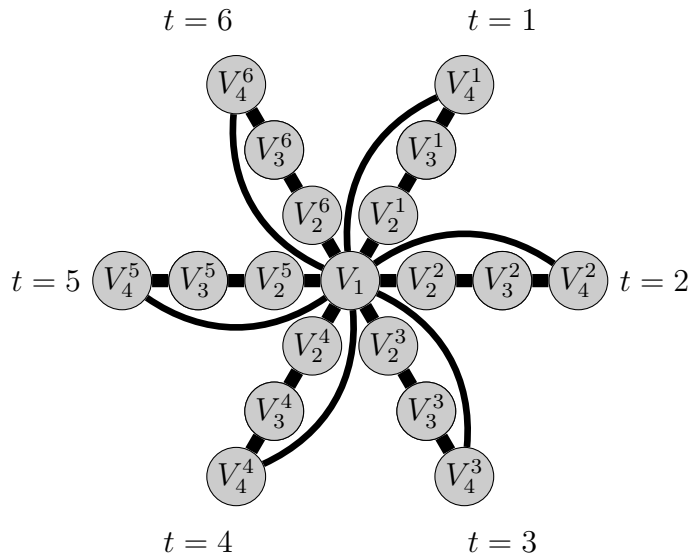
\begin{figure}[!b]
    \centering
   \begin{tikzpicture}
   
    \tikzstyle{vertex}=[circle, draw, fill=black!20,
                         inner sep=0pt, minimum width=22pt];
    \node[vertex] (v00) at (0,0){$V_1$};
   
    \foreach \i in {1,...,6}
    {
    \foreach \j in {2,3}
    {
   
    \node[vertex] (v\i\j) at (120-\i*360/6:\j-1){$V^\i_{\j}$} ;
    \def\k{\the\numexpr\j+1}    
    \node[vertex] (v\i) at (120-\i*360/6:\j){$V^\i_{\k}$} ;
      \draw[line width=5pt] (v\i\j)--(v\i);
    } 
    \draw[line width=5pt] (v00)--(v\i2);
    }
    \foreach \i in {1,...,6}
    {\draw[line width=2.5pt] (v00) to [bend left=35](v\i);
    \node[] (v) at (120-\i*360/6:4) {$t=\i$};
    }

   \end{tikzpicture}
   \caption{The monochromatic configuration from \Cref{Graph} for $q=4,D=6$}
\end{figure}

  Furthermore, by making use of Corollary \ref{angelika2}
 we get that each of these chains is also expanding with respect to $\delta,\gamma, \nu, C$. Notice that the size of the sets in the chain is $\mu n$ and the number of edges between each pair is $(\mu n)^2p'$, which implies that all but at most $(\delta+\gamma)|V_1|$ singletons in $V_1$ are $(1-\nu)$-spanning if $c$ is chosen large enough. Indeed, this is true by Remark \ref{singletons} and the fact that:
\[
C\frac{(\mu n)^{q+1}}{((\mu n)^2p')^{q}}=C\frac{(\mu n)^{q+1}}{(c'c\mu^2n^{1+1/q})^q}=\frac{C}{c'c\mu^{q-1}}<1
\]
where the last inequality is true as $c$ can be chosen large enough in the beginning since $C,c'$ and $\mu$ depend only on $D,q,r$. Note that here it was crucial that we chose $p=\Omega(n^{-1+1/q})$.
\par 
 Now let $B_t=\bigcup_{k=2}^q V_k^t$ for $t\in[D]$.
 We will show the existence a large set $A\subset V_1$ such that for all $t\in [D]$, every $a\in A$ is $(A,B_t,q,\nu)$-expanding.
 First note that if for a fixed $t$ a singleton $\{v\}$ is $(1-\nu)$-spanning in the chain induced by $\{V_1\}\cup\Big\{V_k^t \mid k\in \{2,...,q\}\Big\}$, then $v$ is $(V_1,B_t,q,\nu)$-expanding by definition.
 Therefore, at most $(\gamma+\delta)|V_1|$ vertices in $V_1$ are not $(V_1,B_t,q,\nu)$-expanding for each fixed $t$. By removing these vertices (for each $t$) from $V_1$ we get a set $A\subset V_1$ of vertices which are all $(V_1,B_t,q,\nu)$-expanding.
 
 Since $A$ is of size at least $|V_1|(1-D(\gamma+\delta))=|V_1|(1-1/4)>|V_1|/2$, we also know that each vertex in $A$ is $(A,B_t,q,2\nu)$-expanding, for all $t\in[D]$. This is because for each $v\in A$ the proportion of non-reachable vertices (bounded by $\nu)$ can increase only by a factor of $2$ when we delete at most half of the vertices from $V_1$.
 \par We want to prove that in the monochromatic graph we found it holds that $(A,B_t)$ is $$(q,1/4D,|A|,c)\text{-good, for each }t\in [D].$$
  Then we will be done by Lemma \ref{algorithm} and by setting $\mu_0=\frac{\mu}{2c^{q+4}}$. From the previous discussion we can infer that the second property in Definition \ref{goodness} is satisfied for each pair $(A,B_t)$. The third one follows from the properties of the monochromatic configuration given by Proposition \ref{Graph}:
 \begin{align*}
        |N_{A\cup B_t}(v)|&\leq \cfrac{5}{2}\mu n p'<5|A|p'=5|A|c'cn^{-1+1/q}\\&=5|A|c'c(\mu n)^{-1+1/q}\cdot\mu^{1-1/q}<2|A|c|A|^{-1+1/q}
 \end{align*}
 where for the second inequality we use $\mu n=|V_1|<2|A|$
 and for the last $c'<1/2, \mu\leq 1/3$ and $|A|\leq \mu n$.
 The fourth condition is also satisfied, as it is a consequence of Lemma \ref{Chernoff} with $\alpha=\frac{\mu}{2c^{q+3/2}}$, for $c$ large enough. This completes the proof.
\end{proof}
\vspace{-0.5cm}
\bibliographystyle{unsrt}

\end{document}